\documentclass[conference]{IEEEtran}
\IEEEoverridecommandlockouts

\usepackage{cite}
\usepackage{amsmath,amssymb,amsfonts}
\usepackage{amsthm}
\usepackage{algorithmic}
\usepackage{graphicx}
\usepackage{textcomp}
\usepackage{xcolor}

\usepackage{hyperref}
\usepackage{booktabs}

\newtheorem{theorem}{Theorem}[section]
\newtheorem{lemma}[theorem]{Lemma}

\def\BibTeX{{\rm B\kern-.05em{\sc i\kern-.025em b}\kern-.08em
    T\kern-.1667em\lower.7ex\hbox{E}\kern-.125emX}}
\begin{document}

\title{Spectral Analysis of the \\ Weighted Frobenius Objective}

\author{\IEEEauthorblockN{Vladislav Trifonov}
\IEEEauthorblockA{\textit{AI4S Center, Sberbank of Russia} \\
\textit{CAI, Skoltech}\\
Moscow, Russian Federation \\
vladislav.trifonov@skoltech.ru}
\and
\IEEEauthorblockN{Ivan Oseledets}
\IEEEauthorblockA{\textit{AIRI} \\
\textit{CAI, Skoltech} \\
Moscow, Russia \\
oseledets@airi.net}
\and
\IEEEauthorblockN{Ekaterina Muravleva}
\IEEEauthorblockA{\textit{AI4S Center, Sberbank of Russia} \\
\textit{CAI, Skoltech}\\
Moscow, Russian Federation \\
e.muravleva@skoltech.ru}
}

\maketitle

\begin{abstract}
    We analyze a weighted Frobenius loss for approximating symmetric positive definite matrices in the context of preconditioning iterative solvers. Unlike the standard Frobenius norm, the weighted loss penalizes error components associated with small eigenvalues of the system matrix more strongly. Our analysis reveals that each eigenmode is scaled by the corresponding square of its eigenvalue, and that, under a fixed error budget, the loss is minimized only when the error is confined to the direction of the largest eigenvalue. This provides a rigorous explanation of why minimizing the weighted loss naturally suppresses low-frequency components, which can be a desirable strategy for the conjugate gradient method. The analysis is independent of the specific approximation scheme or sparsity pattern, and applies equally to incomplete factorizations, algebraic updates, and learning-based constructions. Numerical experiments confirm the predictions of the theory, including an illustration where sparse factors are trained by a direct gradient updates to IC(0) factor entries, i.e., no trained neural network model is used.
\end{abstract}

\begin{IEEEkeywords}
    Preconditioning; Conjugate Gradient Method; Sparse Matrix Approximations; Spectral Properties
\end{IEEEkeywords}

\section{Motivation and background}

We are concerned with the construction of structured approximations to large sparse matrices in the context of preconditioning. Preconditioners of this type are widely used to accelerate iterative methods for solving linear systems, especially those arising from discretizations of partial differential equations. While our focus is on preconditioners, the considerations developed here may also be relevant for other approximation problems where a structured approximation of a matrix is required.

A common procedure is to select a structured matrix $P$ that minimizes the Frobenius functional~[1]
\begin{equation}
    \label{eq:unweighted-frobenius}
\|P - A\|_F^2,
\end{equation}
with $A$ the given matrix. This classical choice is orthogonally invariant and computationally convenient. At the same time, this functional assigns equal importance to all components of the error, without regard to the spectral properties of $A$.

In this paper we investigate the weighted Frobenius functional
\begin{equation}
    \label{eq:weighted-frobenius}
\|(P - A)A^{-1}\|_F^2,
\end{equation}
where $A$ is assumed symmetric positive definite and $P$ is again restricted to a structured class, e.g., sparse factors with a prescribed sparsity pattern of incomplete Cholesky factorizations with zero level fill-in (IC(0)). Since $A$ is SPD, the eigendecomposition $A = Q \Lambda Q^\top$ exists. Expressing the error in this eigenbasis allows a precise analysis of how the two functionals differ.

The assumption that $A$ is symmetric positive definite is not restrictive in many applications. Such matrices arise naturally as stiffness matrices from finite element or finite difference discretizations of elliptic partial differential equations, as covariance matrices in statistics, and as normal equations in least-squares problems. In the PDE setting, the smallest eigenvalues are typically associated with smooth, slowly varying error modes, while the largest eigenvalues correspond to oscillatory components introduced by the discretization grid. The distinction between these modes is central in numerical analysis, as different iterative methods and preconditioners affect them in different ways.

Approximations $P$ with prescribed sparsity or structure are frequently used as preconditioners to accelerate iterative solvers for these systems. Classical examples include IC(0), which maintain the sparsity of the original matrix, and algebraic approximations designed to capture the dominant spectral behavior of $A$ while remaining computationally efficient. In such contexts, the choice of functional used to determine $P$ directly influences which parts of the spectrum are emphasized by the approximation.

\section{Spectral properties of Frobenius functionals}

We now analyze the difference between the unweighted functional~\eqref{eq:unweighted-frobenius}
and the weighted version~\eqref{eq:weighted-frobenius}, where $A$ is symmetric positive definite and $P$ is a structured approximation. The following lemma provides a spectral decomposition of both functionals.

\begin{lemma}[Spectral decomposition]\label{lem:spectral-objectives}
Let $A = Q \Lambda Q^\top$ be the eigen-decomposition of a symmetric positive definite matrix 
$A \in \mathbb{R}^{n \times n}$.
For any matrix $P$, set $E := P - A$ and define
\begin{equation*}
a_j := \sum_{i=1}^n \big|(Q^\top E Q)_{ij}\big|^2, \qquad j=1,\dots,n.
\end{equation*}
Then
\begin{equation*}
\|E\|_F^2 = \sum_{j=1}^n a_j,
\qquad
\|E A^{-1}\|_F^2 = \sum_{j=1}^n \frac{a_j}{\lambda_j^2}.
\end{equation*}
\end{lemma}

\begin{proof}
Since $A^{-1} = Q \Lambda^{-1} Q^\top$, we have
\begin{equation*}
E A^{-1} = Q (Q^\top E Q) \Lambda^{-1} Q^\top.
\end{equation*}
By invariance of the Frobenius norm under orthogonal transformations,
\begin{equation*}
\|E\|_F^2 = \|Q^\top E Q\|_F^2 = \sum_{j=1}^n \sum_{i=1}^n |(Q^\top E Q)_{ij}|^2 = \sum_{j=1}^n a_j
\end{equation*}
and
\begin{align*}
\|E A^{-1}\|_F^2 &= \|(Q^\top E Q)\Lambda^{-1}\|_F^2 \nonumber \\
&= \sum_{j=1}^n \sum_{i=1}^n \frac{|(Q^\top E Q)_{ij}|^2}{\lambda_j^2} = \sum_{j=1}^n \frac{a_j}{\lambda_j^2}
\end{align*}
\end{proof}

\begin{theorem}[Weighted functional inequality]
For $A$, $E$, and $a_j$ as in Lemma~\ref{lem:spectral-objectives}, assume that 
$\sum_{j=1}^n a_j = c > 0$. Then
\begin{equation*}
\|E A^{-1}\|_F^2 = \sum_{j=1}^n \frac{a_j}{\lambda_j^2} \;\;\ge\;\; \frac{c}{\lambda_n^2},
\end{equation*}
with equality if and only if $a_j = 0$ for all $j < n$ and $a_n = c$.
\end{theorem}

\begin{proof}
Because $\lambda_1^{-2} \ge \cdots \ge \lambda_n^{-2}$ and the mapping 
$(a_1,\dots,a_n) \mapsto \sum_j a_j/\lambda_j^2$ is linear, 
the minimum under the constraint $\sum_j a_j = c$ is achieved by placing all energy 
at the smallest weight $\lambda_n^{-2}$. This proves the inequality and the 
corresponding condition for equality.
\end{proof}

\section{Computational experiments}

Recent advances in machine learning have explored neural network-based approaches for preconditioning, particularly Graph Neural Networks (GNNs) for constructing sparse factorized preconditioners~[3, 5]. However, recent studies~[4] have demonstrated fundamental limitations of message-passing GNNs in approximating non-local sparse triangular factorizations.

Our goal is to explore whether preconditioners can be improved by directly optimizing their sparse factors with gradient descent on different objectives. We investigate if this parameter-free approach can lead to better spectral properties for IC(0) factors with classical algorithm.

Our experiments aimed to demonstrate the effect on the spectral properties of the preconditioner for a single matrix. While this example is limited, many real-world applications require solving linear systems with a single coefficient matrix $A$ but multiple right-hand sides $\{ b_i\}_{i=1}^N$. This scenario is common in:

\begin{itemize}
\item Time-dependent PDEs, where the system matrix remains constant while the right-hand side evolves over time.
\item Parameter studies in optimization, where parameters change but the system structure remains fixed.
\item Multiple load cases in structural analysis, where different loading conditions are applied to the same structure.
\item Monte Carlo simulations, where random variations affect the right-hand side but not the system matrix.
\end{itemize}

Moreover, this approach operates globally on the entire preconditioner structure, which does not suffer from the principal limitations of GNNs for approximating sparse triangular factorizations.

To be more specific, we directly apply gradient updates of the losses to the entries of the IC(0) factors. We initialize training with IC(0) and update only the entries in its prescribed sparsity pattern. In practice, starting from IC(0) and using a small learning rate preserved a positive diagonal and yielded a usable preconditioner throughout.

The gradient of the loss function with respect to the factor entries $L_{ij}$ is computed using automatic differentiation. The update rule follows:
\begin{equation*}
L_{ij}^{(k+1)} = L_{ij}^{(k)} - \alpha \frac{\partial \mathcal{L}}{\partial L_{ij}}
\end{equation*}
where $\mathcal{L}$ is either loss~\eqref{eq:unweighted-frobenius} or loss~\eqref{eq:weighted-frobenius}, $\alpha$ is the learning rate and $k$ denotes the iteration number. Both losses are computed using Hutchinson's stochastic trace estimator~[2] as described in~[5].

We consider $4096\times 4096$ SPD matrix from a cell-centered second-order finite-volume discretization of
\begin{equation*}
-\nabla \cdot \big(k(x, y)\nabla u(x,y)\big) = f(x,y)
\end{equation*}
on a unit square with Dirichlet boundary conditions. The coefficient field $k(x,y)$ is the exponential of a Gaussian random field (global contrast $\approx10$). We draw
$10^3$ random right-hand sides from $\mathcal{N}(0,I)$. Other hyperparameters are: learning rate $10^{-3}$, batch of size $512$, and $10^4$ epochs.

We evaluate three characteristics: (i) the condition number $\kappa(P^{-1}A)$; (ii) CG iteration counts to a fixed relative residual tolerance; and (iii) eigenvalue histograms of $P^{-1}A$. These characteristics consistently indicate that optimizing IC(0) with the weighted objective~\eqref{eq:weighted-frobenius} shifts error away from small-eigenvalue modes, improving conditioning and CG convergence.

\begin{figure*}[ht]
    \vskip 0.2in
    \begin{center}
    \centerline{\includegraphics[width=1.\textwidth]{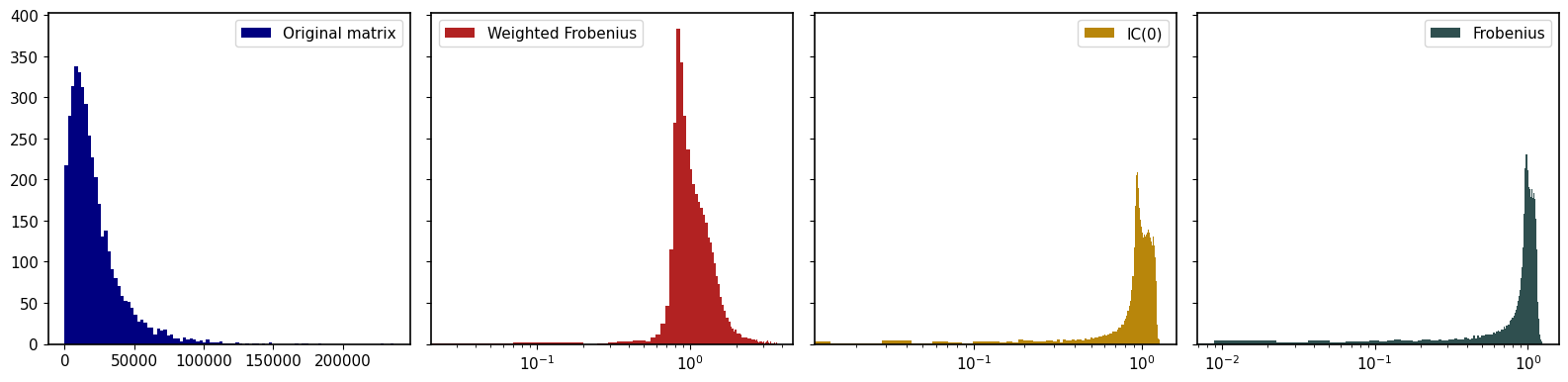}}
    \caption{Distributions of eigenvalues of the different matrices. From left to right: initial matrix $A$; preconditioned matrix $P^{-1}A$ with weighted objective~\eqref{eq:weighted-frobenius}; with IC(0); and with unweighted objective~\eqref{eq:unweighted-frobenius}.}
    \label{fig:eigenvalues_distribution}
    \end{center}
    \vskip -0.2in
\end{figure*}

\section{Discussion}

As shown below, weighting the loss by $A^{-1}$ yields a better conditioned $P^{-1}A$ than both IC(0) and the unweighted objective. Table~\ref{tab:condition_number} illustrates the condition number of the preconditioned system with different matrices. \textit{$A$} - initial system matrix. \textit{Weighted Frobenius} - initial matrix preconditioned with factors trained with weighted Frobenius objective~\eqref{eq:weighted-frobenius}; \textit{IC(0)} - initial matrix preconditioned with factors from IC(0); \textit{Frobenius} - initial matrix preconditioned with factors trained with unweighted Frobenius objective~\eqref{eq:unweighted-frobenius}. Numerically, $\kappa(P^{-1}A)$ drops from $11802$ (no preconditioner) to $232$ with IC(0), to $302$ with the unweighted objective, and to $130$ with the weighted objective, i.e., a $\sim1.8\times$ improvement over IC(0) and $\sim2.3\times$ over the unweighted variant.

\begin{table}[ht]
    \caption{Condition number $\kappa(P^{-1}A)$ for different preconditioners: IC(0) baseline, and factors optimized with the unweighted (1) vs. weighted (2) Frobenius objectives.}
    \label{tab:condition_number}
    \centering
    \begin{tabular}{cccc} 
        \toprule
        $A$ & Frobenius & Weighted Frobenius & IC(0) \\
        \midrule
        $11802$ & $302$ & $130$ & $232$ \\
        \bottomrule
    \end{tabular}
\end{table}

The Fig.~\ref{fig:eigenvalues_distribution} illustrates the eigenvalue distributions of the different matrices. Fig.~\ref{fig:cg_iterations} demonstrates CG method convergence with different preconditioners. Fig.~\ref{fig:loss_history} shows the normalized loss convergence during training. The weighted objective yields a tighter eigenvalue cluster for $P^{-1}A$ with fewer extreme small eigenvalues, explaining the lower condition number in Table~\ref{tab:condition_number}.

\begin{figure}[ht]
\vskip 0.2in
\begin{center}
    \centerline{\includegraphics[width=.8\columnwidth]{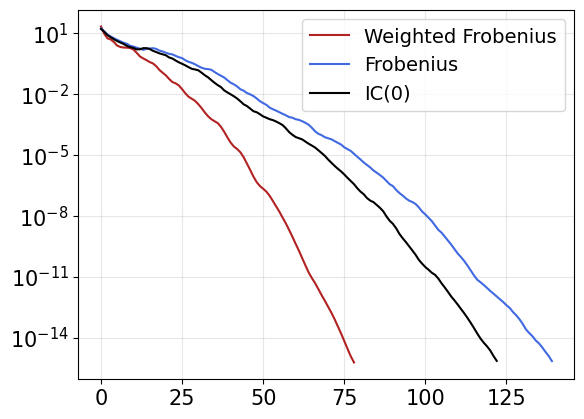}}
    \caption{CG convergence on a single linear system. Relative residual vs. iterations; same right-hand side for each method. \textit{Weighted Frobenius:} preconditioner optimized with~\eqref{eq:weighted-frobenius}; \textit{IC(0):} baseline incomplete Cholesky; \textit{Frobenius:} preconditioner optimized with~\eqref{eq:unweighted-frobenius}.}
    \label{fig:cg_iterations}
\end{center}
\vskip -0.2in
\end{figure}

\begin{figure}[ht]
    \vskip 0.2in
    \begin{center}
    \centerline{\includegraphics[width=.8\columnwidth]{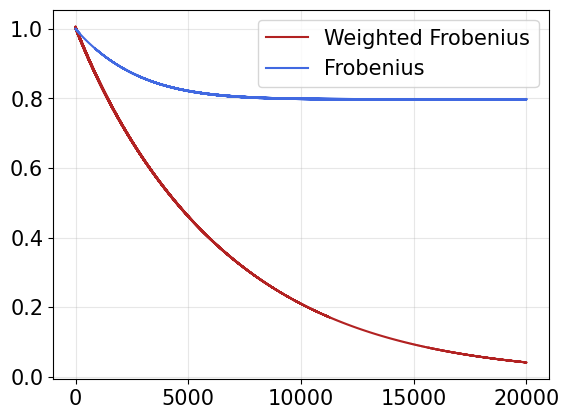}}
    \caption{Normalized training loss vs. update steps. Average over stochastic trace batches. \textit{Weighted Frobenius:} objective~\eqref{eq:weighted-frobenius}; \textit{Frobenius:} objective~\eqref{eq:unweighted-frobenius}.}
    \label{fig:loss_history}
    \end{center}
    \vskip -0.2in
\end{figure}

Theoretical analysis shows that the weighted Frobenius functional assigns larger penalties to error components aligned with small eigenvalues of $A$. In applications where $A$ arises from discretizations of elliptic PDEs, these directions often correspond to smooth or low-frequency modes. Thus, minimizing $\|(P - A)A^{-1}\|_F^2$ enforces stronger suppression of such components compared to the plain Frobenius functional. This effect is independent of how the approximation $P$ is parameterized. 

It is important to note the scope of these observations. The algebraic decomposition is method-agnostic: it describes how the two functionals weight different spectral components of the error, independently of the solver that will use the approximation. The implications for convergence depend on the choice of iterative method and on how the approximation is used. For CG on SPD systems, sensitivity to small eigenvalues makes the connection to our analysis direct. For GMRES or other Krylov subspace methods on nonsymmetric problems, convergence depends on the singular value distribution, the field of values, or the pseudospectrum of the preconditioned matrix. In such settings, suppressing components aligned with small singular values can be beneficial, but it does not constitute a universal guarantee of improved convergence. 

In summary, the weighted Frobenius functional provides a clear, method-independent description of how spectral components of the error are emphasized in structured approximations. Overall, a spectrally weighted Frobenius objective offers a method-agnostic lever to emphasize the right error modes. Its solver-level benefit is clearest for CG on SPD systems and remains problem-dependent elsewhere.

\end{document}